\documentclass{article}
\usepackage{amsfonts}
\usepackage{amsmath}
\usepackage{color}
\usepackage{amsfonts,amssymb,amsmath,amsthm}
\usepackage{hyperref}
\usepackage{geometry}
\usepackage{graphicx}

\newtheorem{theorem}{Theorem}

\newtheorem{corollary}[theorem]{Corollary}

\newtheorem{lemma}[theorem]{Lemma}

\begin{document}

\title{Subcritical branching processes in random environment with
immigration: survival of a single family}
\author{E.E.Dyakonova\thanks{%
Department of Discrete Mathematics, Steklov Mathematical Institute of
Russian Academy of Sciences, 8 Gubkin Street, 117 966 Moscow GSP-1, Russia
E-mail: elena@mi-ras.ru},\, V. A. Vatutin\thanks{%
Department of Discrete Mathematics, Steklov Mathematical Institute of
Russian Academy of Sciences, 8 Gubkin Street, 117 966 Moscow GSP-1, Russia,
E-mail: vatutin@mi-ras.ru}}
\date{}
\maketitle

\begin{abstract}
We consider a subcritical branching process in an i.i.d. random environment,
in which one immigrant arrives at each generation. We consider the event $%
\mathcal{A}_{i}(n)$ that all individuals alive at time $n$ are offspring of
the immigrant which joined the population at time $i$ and investigate the
asymptotic probability of this extreme event when $n\to\infty$ and $i$ is
either fixed, or the difference $n-i$ is fixed, or $\min(i,n-i)\to\infty.$
To deduce the desired asymptotics we establish some limit theorems for
random walks conditioned to be nonnegative or negative.\newline


\end{abstract}

\section{Introduction and main results}

We consider a branching process with immigration evolving in a random
environment. Individuals in such a process reproduce independently of each
other according to random offspring distributions which vary from one
generation to the other. In addition, an immigrant enters the population at
each generation. Denote by $\Delta $ the space of all probability measures
on $\mathbb{N}_{0}:=\{0,1,2,\ldots \}.$ Equipped with the metric of total
variation $\Delta $ becomes a Polish space. We specify on the Borel $\sigma$%
-algebra of $\Delta $ a probability measure $\mathbf{P}$.

Let $F$ be a random variable taking values in $\Delta $, and let $F_{n},n\in
\mathbb{N}:=\mathbb{N}_{0}\backslash \left\{ 0\right\} $ be a sequence of
independent copies of $F $. The infinite sequence $\mathcal{E}=\left\{
F_{n},n\in \mathbb{N}\right\} $ is called a random environment.

A sequence of $\mathbb{N}_{0}$-valued random variables $\mathbf{Y}=\left\{
Y_{n},\ n\in \mathbb{N}_{0}\right\} $ specified on a probability space $%
(\Omega ,\mathcal{F},\mathbf{P})$ is called a branching process with one
immigrant in random environment (BPIRE), if $Y_{0}=1$ and, given the
environment, the process $\mathbf{Y}$ is a Markov chain with
\begin{equation}
\mathcal{L}\left( Y_{n}|Y_{n-1}=y_{n-1},F_{i}=f_{i},i\in \mathbb{N}\right) =%
\mathcal{L}(\xi _{n1}+\ldots +\xi _{ny_{n-1}}+1)  \label{BasicDefBPimmigr}
\end{equation}%
for every $n\in \mathbb{N}$, $y_{n-1}\in \mathbb{N}_{0}$ and $%
f_{1},f_{2},...\in \Delta $, where $\xi _{n1},\xi _{n2},\ldots $ are i.i.d.
random variables with distribution $f_{n}.$ Thus, $Y_{n-1}$ is the $(n-1)$th
generation size of the population and $f_{n}$ is the offspring distribution
of an individual at generation $n-1$. It will be convenient to consider that
if $Y_{n-1}=y_{n-1}>0$ is the population size of the $(n-1)$th generation of
$\mathbf{Y}$ then first $\xi _{n1}+\ldots +\xi _{ny_{n-1}}$ individuals of
the $n$th generation are born and afterwards one immigrant enters the
population.

We will call an $(i,n)$-clan the set of individuals alive at generation $n$
and being children of the immigrant which entered the population at
generation $i$. We say that only the $(i,n)$-clan survives in $\mathbf{Y}$
at moment $n$ if $Y_{n}^{-}:=\xi _{n1}+\ldots +\xi _{ny_{n-1}}>0$ and all $%
Y_{n}^{-}$ particles belong to the $(i,n)$-clan. Let $\mathcal{A}_{i}(n)$ be
the event that only the $(i,n)$-clan survives in $\mathbf{Y}$ at moment $n$.
The aim of this paper is to study the asymptotic behavior of the probability
$\mathbf{P}\left( \mathcal{A}_{i}(n)\right) $ as $n\rightarrow\infty $ and $%
i $ varies with $n$ in an appropriate way for subcritical BPIRE's.

The problem we consider admits the following biological interpretation.
Assume, for instance, that each immigrant has a new type or belongs to a new
species. Then the realisation of the event $\mathcal{A}_{i}(n)$ means that
all individuals of the population existing at moment $n$ are offspring of
the immigrant entering the population at moment $i.$ This is a reflection of
the low genetic diversity of the population arising in the course of
evolution. For the critical BPIRE the probability of the event $\mathcal{A}%
_{i}(n)$ has been investigated in \cite{SV2019}.

Branching processes in random environment with one immigrant in each
generation were first analysed in the classical paper~\cite{KZS75} in
connection with studying properties of random walks in random environment.
Later on the model of BPIRE was used in different situations in \cite%
{Afan2012}, \cite{Afan2016}, \cite{Afan2016b}, \cite{Afan2018}. Note also
that the authors of \cite{DLVZ19} and \cite{LVZ19} have studied the tail
distribution of the life-periods of BPRIE's in the critical and subcritical
cases with an immigration law more general than in our case.

We consider, along with the process $\mathbf{Y}$, a standard branching
process $\mathbf{Z}=\left\{ Z_{n},\ n\in \mathbb{N}_{0}\right\} $ in the
random environment (BPRE) which, given $\mathcal{E}$ is a Markov chain with $%
Z_{0}=1$ and
\begin{equation}
\mathcal{L}\left( Z_{n}|Z_{n-1}=z_{n-1},F_{i}=f_{i},i\in \mathbb{N}\right) =%
\mathcal{L}(\xi _{n1}+\ldots +\xi _{nz_{n-1}})  \label{BPordinary}
\end{equation}%
for $n\in \mathbb{N}$, $z_{n-1}\in \mathbb{N}_{0}$ and $f_{1},f_{2},...\in
\Delta $.

To formulate the main results of the paper we introduce the so-called
associated random walk $\mathbf{S}=\left\{ S_{n},n\in \mathbb{N}_{0}\right\}
$. This random walk has increments $X_{n}=S_{n}-S_{n-1}$, $n\geq 1$, defined
as
\begin{equation*}
X_{n}=\log m\left( F_{n}\right) ,
\end{equation*}%
which are i.i.d. copies of the logarithmic mean offspring number $X:=\log $ $%
m(F)$ with%
\begin{equation*}
m(F):=\sum_{j=1}^{\infty }jF\left( \left\{ j\right\} \right) .
\end{equation*}

We associate with each measure $F$ the respective probability generating
function
\begin{equation*}
F(s):=\sum_{j=0}^{\infty }F\left( \left\{ j\right\} \right) s^{j}.
\end{equation*}%
We assume that the random probability generating function meets the
following restrictions.\newline

\noindent \textbf{Hypothesis A1}. The generating function $F(s)$ is
geometric with probability 1, that is%
\begin{equation}
F(s)=\frac{q}{1-ps}=\frac{1}{1+m(F)(1-s)}  \label{Frac_generating}
\end{equation}%
with random $p,q\in (0,1)$ satisfying $p+q=1$ and
\begin{equation*}
m(F)=\frac{p}{q}=e^{\log (p/q)}=e^{X}.
\end{equation*}

\noindent\ The BPRE\ is subcritical, i.e.
\begin{equation}
-\infty <\mathbf{E}X<0  \label{subcritical con}
\end{equation}%
and either $-\infty <\mathbf{E}\left[ Xe^{X}\right] <0$ (the strongly
subcritical case), or $\mathbf{E}\left[ Xe^{X}\right] =0$ (the intermediate
subcritical case), or there is a number $0<\beta <1$ such that
\begin{equation}
\mathbf{E}[Xe^{\beta X}]=0.  \label{weakly condi}
\end{equation}%
(the weakly subcritical case).

Note that the BPRE's mentioned in Hypothesis A1 do not exhaust all possible
cases of subcritical BPRE's. For instance, they do not include the
subcritical BPRE's where $\mathbf{E}\left[ Xe^{tX}\right] =\infty $ for all $%
t>0$ (see \cite{VZ2012}) or where $\mathbf{E}\left[ Xe^{tX}\right] <0$ for
all $0\leq t\leq \beta $ with $\beta =\sup \{t\geq 0:\mathbf{E}\left[ Xe^{tX}%
\right] <\infty \}\in (0,1)$ (see \cite{BV2017}).

One of the main tools in analyzing properties of BPRE and BPIRE is a change
of measure. We follow this approach and introduce a new measure $\mathbb{P}$
by setting, for any $n\in \mathbb{N}$ and any measurable bounded function $%
\psi :\Delta ^{n}\times \mathbb{N}_{0}^{n+1}\rightarrow \mathbb{R}$
\begin{equation}
\mathbb{E}[\psi (F_{1},\cdot \cdot \cdot ,F_{n},Y_{0},\cdot \cdot \cdot
,Y_{n})]:=\gamma ^{-n}\mathbf{E}[\psi (F_{1},\cdot \cdot \cdot
,F_{n},Y_{0},\cdot \cdot \cdot ,Y_{n})e^{\delta S_{n}}],
\label{Change_delta}
\end{equation}%
with
\begin{equation*}
\gamma :=\mathbf{E}[e^{\delta X}],
\end{equation*}%
where $\delta =1$ for strongly and intermediate subcritical BPIRE and $%
\delta =\beta $ for weakly subcritical BPIRE.

Observe that $\mathbf{E}[Xe^{\delta X}]=0$ translates into
\begin{equation*}
\mathbb{E}[X]=0.
\end{equation*}

\noindent \textbf{Hypothesis A2}. If a BPIRE\ is either intermediate or
weakly subcritical then the distribution of $X$ is nonlattice and belongs
with respect to $\mathbb{P}$ to the domain of attraction of a two-sided
stable law with index\textbf{\ }$\alpha \in (1,2]$.

Since $\mathbb{E}[X]=0$, Hypothesis A2 provides existence of an increasing
sequence of positive numbers
\begin{equation}
c_{n}=n^{1/\alpha }l_{1}(n)  \label{Def_a}
\end{equation}%
with slowly varying sequence $l_{1}(1),l_{1}(2),...$ such that, the
distribution law of $S_{n}/c_{n}$ converges weakly, as $n\rightarrow \infty $
to the mentioned two-sided stable law. Besides, under this condition there
exists a number $\rho \in (0,1)$ such that%
\begin{equation}
\lim_{n\rightarrow \infty }\mathbb{P}\left( S_{n}>0\right) =\rho .
\label{Def-ro}
\end{equation}

Recall that $\mathcal{A}_{i}(n)$ is the event that only the $(i,n)$-clan
survives in $\mathbf{Y}$ at moment $n$.

We first consider the strongly subcritical case.

\begin{theorem}
\label{T_strongly}Let $\mathbf{Y}$ be a strongly subcritical BPIRE
satisfying Hypotheses A1. Then

1) for any fixed $N$%
\begin{equation*}
\lim_{n\rightarrow \infty }\mathbf{P}\left( \mathcal{A}_{n-N}(n)\right)
=:r_{N}\in \left( 0,\infty \right) ;
\end{equation*}

2) there exists a constant $R\in \left( 0,\infty \right) $ such that%
\begin{equation*}
\lim_{n-i\rightarrow \infty }\gamma ^{-(n-i)}\mathbf{P}\left( \mathcal{A}%
_{i}(n)\right) =R.
\end{equation*}
\end{theorem}

The next theorem deals with intermediate subcritical case.

\begin{theorem}
\label{T_intermediate}Let $\mathbf{Y}$ be an intermediate subcritical BPIRE
meeting Hypotheses A1 and A2. Then

1) for any fixed $N$%
\begin{equation*}
\lim_{n\rightarrow \infty }\mathbf{P}\left( \mathcal{A}_{n-N}(n)\right)
=:r_{N}\in \left( 0,\infty \right) ;
\end{equation*}

2) there exist a slowly varying function $l(n)$ and a constant $R\in \left(
0,\infty \right)$ such that
\begin{equation*}
\lim_{n-i\rightarrow \infty }\gamma ^{-(n-i)}\left( n-i\right) ^{\rho }l(n-i)%
\mathbf{P}\left( \mathcal{A}_{i}(n)\right) =R .
\end{equation*}
\end{theorem}

The constants $c_n$ defined in (\ref{Def_a} play an important role in the
statement of our third theorem.

\begin{theorem}
\label{T_weakly}Let $\mathbf{Y}$ be a weakly subcritical BPIRE meeting
Hypotheses A1 and A2. Then

1) for any fixed $N$%
\begin{equation*}
\lim_{n\rightarrow \infty }\mathbf{P}\left( \mathcal{A}_{n-N}(n)\right)
=r_{N}\in \left( 0,\infty \right) ;
\end{equation*}

2) for any fixed $i$ there exists a constant $R_i\in \left( 0,\infty \right)$
such that
\begin{equation*}
\lim_{n-i\rightarrow \infty }\gamma ^{-(n-i)}\left( n-i\right) c_{n-i}%
\mathbf{P}\left( \mathcal{A}_{i}(n)\right) =R_i.
\end{equation*}

3) there exists a constant $R\in \left( 0,\infty \right)$ such that
\begin{equation*}
\lim_{\min(i,n-i)\rightarrow \infty }\gamma ^{-(n-i)}\left( n-i\right)
c_{n-i}\mathbf{P}\left( \mathcal{A}_{i}(n)\right) =R .
\end{equation*}
\end{theorem}

The rest of the paper is organised as follows. In Section \ref{sec_aux_res}
we collect some auxiliary results dealing with explicit expressions for the
probability of the event $\mathcal{A}_{i}(n)$ and prove two conditional
limit theorems for random walks conditioned to stay nonnegative or negative.
Section \ref{Sec_NT} contains the proof of a limit theorem for certain
functionals constructed by driftless random walks. Section \ref{sec_nt2} is
dedicated to the proofs of Theorems \ref{T_strongly}--\ref{T_weakly}.

In the sequel we will denote by $C,C_{1}, C_{2},...$ constants which may
vary from line to line and by $K, K_{1},K_{2},...$ some fixed constants.

\section{Auxiliary results}

\label{sec_aux_res}

\subsection{Some identities}

Given the environment $\mathcal{E}=\left\{ F_{n},n\in \mathbb{N}\right\} $,
we construct the i.i.d. sequence of generating functions
\begin{equation*}
F_{n}(s):=\sum_{j=0}^{\infty }F_{n}\left( \left\{ j\right\} \right)
s^{j},\quad s\in \lbrack 0,1],
\end{equation*}%
and use below the convolutions of $F_{1},...,F_{n}$ specified for $0\leq
i\leq n-1$ by the equalities
\begin{equation*}
\left\{
\begin{array}{l}
F_{i,n}(s):=F_{i+1}(F_{i+2}(\ldots F_{n}(s)\ldots )),\quad \\
F_{n,i}(s):=F_{n}(F_{n-1}(\ldots F_{i+1}(s)\ldots )),%
\end{array}%
\right.
\end{equation*}%
and $F_{n,n}(s):=s$ for $i=n$.

Then we can express the probability of the event $\mathcal{A}_i(n)$
conditionally on the random walk $\mathbf{S}$ as follows:
\begin{equation}  \label{expr_PAin}
\mathbf{P}\left( \mathcal{A}_{i}(n)| \mathbf{S}\right) =\mathbf{E}\left[
(1-F_{i,n}(0))\prod_{k\neq i}^{n-1}F_{k,n}(0) \Bigg| \mathbf{S} \right] .
\end{equation}

For the sake of readability, put

\begin{equation*}
\mathfrak{h}_{n}(s):=(1-F_{0,n}(s))\prod_{k=1}^{n-1}F_{k,n}(s),
\end{equation*}%
and, for $0\leq i\leq n$ introduce the notation
\begin{eqnarray*}
&&a_{i,n}:=e^{S_{i}-S_{n}},\qquad a_{n}:=a_{0,n}=e^{-S_{n}} \\
&&b_{i,n}:=\sum_{k=i}^{n-1}e^{S_{i}-S_{k}},\quad
b_{n}:=b_{0,n}=\sum_{k=0}^{n-1}e^{-S_{k}}.
\end{eqnarray*}

We have the following equality:

\begin{lemma}
\label{L_represent2} Under Hypothesis $A1$%
\begin{equation*}
\mathfrak{h}_{n}(s)=\frac{1}{a_{n}\left( 1-s\right) ^{-1}+b_{n}}\frac{%
a_{n}\left( 1-s\right) ^{-1}}{a_{n}\left( 1-s\right) ^{-1}+b_{n}-b_{1}}.
\end{equation*}
\end{lemma}

\textbf{Proof.}
\begin{equation}
\ F_{i}(s)=\frac{q_{i}}{1-p_{i}s}=\frac{1}{1+e^{X_{i}}\left( 1-s\right) }
\end{equation}%
for all $i\in \mathbf{N}$. By induction we can prove that
\begin{equation}
F_{0,n}(s)=1-\frac{1}{a_{n}\left( 1-s\right) ^{-1}+b_{n}}  \label{Frac_F}
\end{equation}%
and, therefore,%
\begin{eqnarray}
F_{i,n}(s) &=&1-\frac{1}{a_{i,n}\left( 1-s\right) ^{-1}+b_{i,n}}  \notag \\
&=&1-\frac{a_{i}}{a_{n}\left( 1-s\right) ^{-1}+b_{n}-b_{i}}=\frac{%
a_{n}\left( 1-s\right) ^{-1}+b_{n}-b_{i+1}}{a_{n}\left( 1-s\right)
^{-1}+b_{n}-b_{i}}.  \label{expr_Fin1}
\end{eqnarray}%
Thus,%
\begin{eqnarray*}
\mathbf{h}_{n}(s) &=&\frac{1}{a_{n}\left( 1-s\right) ^{-1}+b_{n}}%
\prod_{j=1}^{n-1}\frac{a_{n}\left( 1-s\right) ^{-1}+b_{n}-b_{j+1}}{%
a_{n}\left( 1-s\right) ^{-1}+b_{n}-b_{j}} \\
&=&\frac{1}{a_{n}\left( 1-s\right) ^{-1}+b_{n}}\frac{a_{n}\left( 1-s\right)
^{-1}}{a_{n}\left( 1-s\right) ^{-1}+b_{n}-b_{1}}.
\end{eqnarray*}%
This ends the proof.

To conclude this section, we will provide an expression in terms of $a_{i}$%
's and $b_{i}$'s for the random variable
\begin{equation*}
\mathcal{H}_{i,n}:=\left( 1-F_{i,n}(0)\right) \prod_{j\neq
i}^{n-1}F_{j,n}(0).
\end{equation*}

\begin{corollary}
\label{C_fractional} Under Hypothesis $A1$ for any $i=1,2,...,n-1$
\begin{equation*}
\mathcal{H}_{i,n}=\frac{a_{i}}{a_{n}+b_{n}-b_{i+1}}\frac{a_{n}}{a_{n}+b_{n}}.
\end{equation*}
\end{corollary}

\begin{proof}
If $i=0$ then the desired statement is a direct consequence of Lemma \ref%
{L_represent2}, as $\mathcal{H}_{0,n}=\mathbf{h}_{n}(0).$ If $i=1,2,...,n-1$
$\ $then\ the needed statement follows from (\ref{expr_Fin1}), by taking $%
s=0:$
\begin{eqnarray*}
\mathcal{H}_{i,n} &:=&\frac{\left( 1-F_{i,n}(0)\right) }{F_{i,n}(0)}%
\prod_{j=0}^{n-1}F_{j,n}(0) \\
&=&\frac{a_{i}}{a_{n}+b_{n}-b_{i+1}}\prod_{k=0}^{n-1}\frac{%
a_{n}+b_{n}-b_{j+1}}{a_{n}+b_{n}-b_{j}}=\frac{a_{i}}{a_{n}+b_{n}-b_{i+1}}%
\frac{a_{n}}{a_{n}+b_{n}}.
\end{eqnarray*}%
.
\end{proof}

\bigskip

\subsection{Measures $\mathbb{P}_{x}^{+}$ and $\mathbb{P}_{x}^{-}$}

The random variables
\begin{equation}
M_{n}:=\max \left( S_{1},...,S_{n}\right) ,\quad L_{j,n}:=\min \left(
S_{j},S_{j+1},...,S_{n}\right) ,\quad L_{n}:=L_{0,n}  \label{def_LM}
\end{equation}%
and the moment of the first minimum on the interval $[0,n]$ of the random
walk $\mathbf{S}:$
\begin{equation}
\tau (n):=\min \{0\leq k\leq n:S_{k}=L_{n}\}  \label{def_tau(n)}
\end{equation}%
play important role in this section.

To go further we need to perform two more changes of measure using the
right-continuous functions $U:\mathbb{R}$ $\rightarrow \lbrack 0,\infty )$
and $V:\mathbb{R}$ $\rightarrow \lbrack 0,\infty )$ specified by%
\begin{equation*}
U(x):=1+\sum_{n=1}^{\infty }\mathbb{P}\left( S_{n}\geq -x,M_{n}<0\right) ,\
x\geq 0;U(x)=0,\ x<0,
\end{equation*}%
\begin{equation*}
V(x):=1+\sum_{n=1}^{\infty }\mathbb{P}\left( S_{n}<-x,L_{n}\geq 0\right) ,\
x\leq 0;V(x)=0,\ x>0.
\end{equation*}

It is known (see, for instance, \cite{4h}\ and \cite{ABKV}) that for any
oscillating random walk
\begin{equation}
\mathbb{E}\left[ U(x+X);X+x\geq 0\right] =U(x),\quad x\geq 0,  \label{Mes1}
\end{equation}%
and
\begin{equation}
\mathbb{E}\left[ V(x+X);X+x<0\right] =V(x),\quad x\leq 0.  \label{Mes2}
\end{equation}

Let $\mathcal{E}=\left\{ F_{1},F_{2},...\right\} $ be a random environment
and let $\mathcal{F}_{n}$ be the $\sigma $-algebra of events generated by
the random sequences $F_{1},F_{2},...,F_{n}$ and $Y_{0},Y_{1},...,Y_{n}$.
The $\sigma $-algebras $\left\{ \mathcal{F}_{n},n\geq 1\right\} $ form a
filtration $\mathfrak{F}$. Clearly, the increment $X_{n},n\geq 1,$ of the
random walk $\mathbf{S}$ is measurable with respect to $\mathcal{F}_{n}$.
Using the martingale properties (\ref{Mes1})-(\ref{Mes2}) of $U$ and $V$ we
introduce, for each $n$ a probability measure $\mathbb{P}_{(n)}^{+}$ on the $%
\sigma $-algebra $\mathcal{F}_{n}$ in a standard way (see, for instance,
\cite{GV2017}, Chapter 7) by means of the density
\begin{equation*}
d\mathbb{P}_{(n)}^{+}:=U(S_{n})I\left\{ L_{n}\geq 0\right\} d\mathbb{P}.
\end{equation*}%
This and Kolmogorov's extension theorem show that, on a suitable probability
space there exists a probability measure $\mathbb{P}^{+}$ on $\mathfrak{F}$
such that
\begin{equation}
\mathbb{P}^{+}|\mathcal{F}_{n}=\mathbb{P}_{(n)}^{+},\ n\geq 1.
\label{DefMeasures}
\end{equation}

In the sequel we allow for arbitrary initial value $S_{0}=x$. Then, we write
$\mathbb{P}_{x}$ and $\mathbb{E}_{x}$ for the corresponding probability
measures and expectations. Thus, $\mathbb{P}=\mathbb{P}_{0}$ and $\mathbb{E}=%
\mathbb{E}_{0}.$ This agreement allows us to rewrite (\ref{DefMeasures}) as
\begin{equation*}
\mathbb{E}_{x}^{+}\left[ O_{n}\right] :=\frac{1}{U(x)}\mathbb{E}_{x}\left[
O_{n}U(S_{n});L_{n}\geq 0\right] ,\ x\geq 0,
\end{equation*}%
for every $\mathcal{F}_{n}$-measurable random variable $O_{n}$.

Similarly, $V$ gives rise to probability measures $\mathbb{P}_{x}^{-},x\leq
0 $, which can be defined via:
\begin{equation*}
\mathbb{E}_{x}^{-}\left[ O_{n}\right] :=\frac{1}{V(x)}\mathbb{E}_{x}\left[
O_{n}V(S_{n});M_{n}<0\right] ,\ x\leq 0.
\end{equation*}

By means of the measures $\mathbb{P}_{x}^{+}$ and $\mathbb{P}_{x}^{-}$, we
investigate the limit behavior of certain conditional distributions.

First we recall some known results concerning properties of the random
variables $M_{n}$ and $L_{n}$.

\begin{lemma}
\label{L_min_max} (see, for instance, Lemma 2.1 in \cite{4h}) If Hypothesis
A2 is valid then there exist a slowly varying function $l_{1}(n)$ and a
constant $\kappa>0$ such that, as $n\rightarrow \infty $
\begin{equation}
\mathbb{P}\left( L_{n}\geq 0\right) \sim \frac{l_{1}(n)}{n^{1-\rho }};\
\mathbb{P}\left( M_{n}<0\right) \sim \frac{\kappa}{n^{\rho }l_{1}(n)}.
\label{AsymMin}
\end{equation}
\end{lemma}

The next lemma describe asymptotic behavior of some conditional functionals.
Denote%
\begin{equation}
\mathfrak{m}_{1\theta }:=\int_{0}^{\infty }e^{-\theta z}U(z)\,dz,\quad
\mathfrak{m}_{2\theta }:=\int_{-\infty }^{0}e^{\theta z}V(z)\,dz
\label{def_ci}
\end{equation}%
for $\theta>0$ and introduce the measures%
\begin{equation*}
u_{\theta }(z):=\mathfrak{m}_{1\theta }^{-1}U(z)\,I\left\{ z\geq 0\right\}
dz,\quad v_{\theta }(z):=\mathfrak{m}_{2\theta }^{-1}V(z)\,I\left\{
z<0\right\} dz.
\end{equation*}

\begin{lemma}
\label{L_Exponent_cond}(see Proposition 2.1 in \cite{ABKV} ) If Hypothesis
A2 is valid then there exists a positive constants $K$ such that for each $%
\theta >0$, as $n\rightarrow \infty $
\begin{equation}
\mathbb{E}_{x}\left[ e^{-\theta S_{n}};L_{n}\geq 0\right] \sim \frac{K}{%
nc_{n}}U(x)\mathfrak{m}_{2\theta },\quad x\geq 0,  \label{AsymConditional}
\end{equation}%
and
\begin{equation}
\mathbb{E}_{x}\left[ e^{\theta S_{n}};\tau (n)=n\right] \sim \frac{K}{nc_{n}}%
V(x)\mathfrak{m}_{1\theta }\quad x\leq 0.  \label{AsymConditional_max}
\end{equation}
\end{lemma}

The next two lemmas are natural modifications of Lemmas 7.3 and 7.5 in \cite%
{GV2017}, Chapter 7. We fix $0<\delta <1$ and use the agreement $\delta
n:=\lfloor \delta n\rfloor $ in their formulations.

\begin{lemma}
\label{L_Newp41}. Let $W_{n}:=w_{n}(F_{1},\ldots ,F_{\delta n})$, $n\in
\mathbb{N}$, be random variables with values in an Euclidean (or Polish)
space $\mathcal{W}$ such that, as $n\rightarrow \infty $
\begin{equation*}
W_{n}\ \rightarrow \ W_{\infty }\quad \mathbb{P}^{+}\text{-a.s.}
\end{equation*}%
for some $\mathcal{W}$-valued random variable $W_{\infty }$. Also let $%
T_{n}:=t_{n}(F_{1},\ldots ,F_{\delta n})$, $n\geq 1$, be random variables
with values in an Euclidean (or Polish) space $\mathcal{T}$ such that, as $%
n\rightarrow \infty $,
\begin{equation*}
T_{n}\ \rightarrow \ T_{\infty }\quad \mathbb{P}_{x}^{-}\text{-a.s. }
\end{equation*}%
for all $x\leq 0$ and some $\mathcal{T}$-valued random variable $T_{\infty
}. $ Denote
\begin{equation*}
\tilde{T}_{n}:=t_{n}(F_{n},\ldots ,F_{n-\delta n+1})\ .
\end{equation*}%
Let, further $\varphi :\mathcal{W}\times \mathcal{T}\times \mathbb{R}%
_{+}\rightarrow \mathbb{R}$, be a continuous function such that
\begin{equation*}
\sup_{\left( u,v,z\right) \in \mathcal{W}\times \mathcal{T}\times \mathbb{R}%
_{+}}\left\vert \varphi (u,v,z)\right\vert e^{\theta z}<\infty
\end{equation*}%
for some $\theta >0$. If Hypothesis A2 is valid then
\begin{align}
\lim_{n\rightarrow \infty }& nc_{n}\mathbb{E}[\varphi (W_{n},\tilde{T}%
_{n},S_{n})\;;\;L_{n}\geq 0]\;  \notag \\
& =\ K\iiint \varphi (u,v,-z)\mathbb{P}^{+}\left( W_{\infty }\in du\right)
\mathbb{P}_{z}^{-}\left( T_{\infty }\in dv\right) V(z)dz.  \label{Cond1}
\end{align}
\end{lemma}

The following lemma is a counterpart.

\begin{lemma}
\label{L_Newp42} Let $W_{n},T_{n},\tilde{T}_{n}$, $n\in \mathbb{N}$, be as
in Lemma \ref{L_Newp41}, now fulfilling, as $n\rightarrow \infty $
\begin{equation*}
W_{n}\ \rightarrow \ W_{\infty }\quad \mathbb{P}_{x}^{+}\text{-a.s.},\quad
T_{n}\ \rightarrow \ T_{\infty }\quad \mathbb{P}^{-}\text{-a.s.}
\end{equation*}%
for all $x\geq 0$. Let, further $\varphi :\mathcal{W}\times \mathcal{T}%
\times \mathbb{R}_{-}\rightarrow \mathbb{R}$ be a continuous function such
that
\begin{equation*}
\sup_{\left( u,v,z\right) \in \mathcal{W}\times \mathcal{T}\times \mathbb{R}%
_{-}}\left\vert \varphi (u,v,z)\right\vert e^{-\theta z}<\infty
\end{equation*}%
for some $\theta >0$. If Hypothesis A2 is valid then
\begin{align}
& \lim_{n\rightarrow \infty }nc_{n}\mathbb{E}[\varphi (W_{n},\tilde{T}%
_{n},S_{n})\;;\;\tau (n)=n]\;  \notag \\
& \qquad =K\iiint \varphi (u,v,-z)\mathbb{P}_{z}^{+}\left( W_{\infty }\in
du\right) \mathbb{P}^{-}\left( T_{\infty }\in dv\right) U(z)dz.
\label{Cond3}
\end{align}
\end{lemma}

\textbf{Proof.} The proofs of the two statements are very similar. We show
only the first one. Since $e^{\theta z}\varphi (x,y,z)$ is a bounded
continuous function, we may apply Lemma 7.3 in \cite{GV2017}, Chapter 7 and
using the definition of $\nu _{\theta }(dz)$ to conclude that, as $%
n\rightarrow \infty $
\begin{align*}
& \frac{\mathbb{E}[\varphi (W_{n},\tilde{T}_{n},S_{n})e^{\theta
S_{n}}e^{-\theta S_{n}}\;;\;L_{n}\geq 0]}{\mathbb{E}[e^{-\theta
S_{n}};L_{n}\geq 0]} \\
& \qquad \qquad \rightarrow \ \iiint e^{-\theta z}\varphi (u,v,-z)\mathbb{P}%
^{+}\left( W_{\infty }\in du\right) \mathbb{P}_{z}^{-}\left( T_{\infty }\in
dv\right) \nu _{\theta }(dz) \\
& \qquad \qquad =\mathfrak{m}_{2\theta }^{-1}\iiint \varphi (u,v,-z)\mathbb{P%
}^{+}\left( W_{\infty }\in du\right) \mathbb{P}_{z}^{-}\left( T_{\infty }\in
dv\right) V(z)\,dz.
\end{align*}%
To complete the proof of the lemma it remains to recall (\ref%
{AsymConditional}).

\section{A limit theorem for the associated random walk\label{Sec_NT}}

Set
\begin{equation*}
B_{j,n}:=b_{n}-b_{j}=\sum_{k=j}^{n-1}e^{-S_{k}},\quad 1\leq j\leq n,
\end{equation*}%
and, for a fix $\delta \in (0,1)$ put

\begin{equation*}
W_{n}:=B_{1,n\delta }=\sum_{k=1}^{\left[ n\delta \right] -1}e^{-S_{k}},\
\tilde{T}_{n}:=\sum_{k=\left[ n\delta \right] }^{n}e^{S_{n}-S_{k}}.
\end{equation*}%
Setting
\begin{equation*}
T_{n}:=\sum_{k=0}^{n-\left[ n\delta \right] }e^{S_{k}}
\end{equation*}%
we conclude by Lemma 2.7 in \cite{4h} that if Hypothesis A2 is valid then,
for any $x\geq 0$ as $n\rightarrow \infty $
\begin{equation}
W_{n}\rightarrow W_{\infty }:=\sum_{k=1}^{\infty }e^{-S_{k}}<\infty \quad
\mathbb{P}_{x}^{+}\text{-a.s.}  \label{Ginf}
\end{equation}%
and
\begin{equation}
T_{n}\rightarrow T_{\infty }:=\sum_{k=0}^{\infty }e^{S_{k}}<\infty \quad
\mathbb{P}_{-x}^{-}\text{-a.s.}  \label{Hinf}
\end{equation}

The next statement is a generalization of a theorem established in \cite%
{guivarc2001proprietes}.

\begin{lemma}
\label{L_Guiv1} Let $g:$ $[0,\infty )$ $\rightarrow \lbrack 0,\infty )$ and $%
h:$ $[0,\infty )\times \lbrack 0,\infty )\rightarrow \lbrack 0,\infty )$ \
be two nonnegative and not identically equal to zero continuous functions
such that, for all $x\geq 0,y\geq 0$
\begin{equation*}
g(x)\leq Cx^{\lambda _{1}},\ h(x,y)\leq \frac{C}{\left( 1+x+y\right)
^{\lambda _{2}}}
\end{equation*}%
for some $0<\lambda _{1}<\lambda _{2}~$\ and a constant $C>0$.

If $X$ belongs with respect to $\mathbb{P}$ to the domain of attraction of a
two-sided stable law with index\textbf{\ }$\alpha \in (1,2]$ then there
exist two positive constants $K_{g,h}$ and $K_{h}$ such that
\begin{equation}
\lim_{n\rightarrow \infty }nc_{n}\mathbb{E}\left[ g(a_{n})h(a_{n},B_{1,n})%
\right] =K_{g,h}\ \   \label{GuivStatement}
\end{equation}%
and
\begin{equation}
\lim_{n\rightarrow \infty }\frac{\mathbb{E}\left[ h(a_{n},B_{1,n})\right] }{%
\mathbb{P}\left( L_{n}\geq 0\right) }=K_{h}.  \label{GuivStatement2}
\end{equation}
\end{lemma}

\textbf{Proof.} We first check the validity of (\ref{GuivStatement}). For a
fixed positive integer $m\in \lbrack 1,n/2]$ we have%
\begin{eqnarray*}
&&\mathbb{E}\left[ g(a_{n})h(a_{n},B_{1,n});\tau (n)\in \lbrack m,n-m]\right]
\\
&&\qquad \leq C_{1}\mathbb{E}\left[ \frac{e^{-\lambda _{1}S_{n}}}{%
(1+\sum_{k=1}^{n}e^{-S_{k}})^{\lambda _{2}}};\tau (n)\in \lbrack m,n-m]%
\right] \\
&&\qquad \leq C_{1}\mathbb{E}\left[ e^{\lambda _{2}S_{\tau (n)}-\lambda
_{1}S_{n}};\tau (n)\in \lbrack m,n-m]\right] \\
&&\qquad =C_{1}\sum_{j=m}^{n-m}\mathbb{E}\left[ e^{(\lambda _{2}-\lambda
_{1})S_{j}+\lambda _{1}(S_{j}-S_{n})};\tau (n)=j\right] \\
&&\qquad =C_{1}\sum_{j=m}^{n-m}\mathbb{E}\left[ e^{(\lambda _{2}-\lambda
_{1})S_{j}};\tau (j)=j\right] \mathbb{E}\left[ e^{-\lambda
_{1}S_{n-j}};L_{n-j}\geq 0\right] \\
&&\qquad =C_{1}\sum_{j=m}^{n-m}\mathbb{E}\left[ e^{(\lambda _{2}-\lambda
_{1})S_{j}};M_{j}<0\right] \mathbb{E}\left[ e^{-\lambda
_{1}S_{n-j}};L_{n-j}\geq 0\right] .
\end{eqnarray*}%
where we have used the duality principle for random walks for the last
transition. By Lemma \ref{L_Exponent_cond} there exist constants $C_{1}$ and
$C_{2}$ such that
\begin{equation*}
\mathbb{E}\left[ e^{(\lambda _{2}-\lambda _{1})S_{n}};M_{n}<0\right] \leq
\frac{C_{1}}{nc_{n}},\quad \mathbb{E}\left[ e^{-\lambda _{1}S_{n}};L_{n}\geq
0\right] \leq \frac{C_{2}}{nc_{n}}.
\end{equation*}%
for all $n.$ Thus, one can find constants $C,C_{1}$ and $C_{2}$ such that,
for all $m\in \lbrack 1,n/2]$%
\begin{eqnarray*}
\sum_{j=m}^{n-m}\mathbb{E}\left[ e^{(\lambda _{2}-\lambda _{1})S_{j}};M_{j}<0%
\right] \mathbb{E}\left[ e^{-\lambda _{1}S_{n-j}};L_{n-j}\geq 0\right] &\leq
&C\sum_{j=m}^{n-m}\frac{1}{jc_{j}}\frac{1}{\left( n-j\right) c_{n-j}} \\
&\leq &C_{1}\frac{2}{nc_{n/2}}\sum_{j=m}^{\infty }\frac{1}{jc_{j}}\leq \frac{%
C_{2}}{nc_{n}},
\end{eqnarray*}%
where we have used (\ref{Def_a}) and properties of regularly varying
finctions to justify the last inequality. By this estimate it is not
difficult to conclude that, for any $\varepsilon >0$ there exists a positive
integer $m=m(\varepsilon )$ such that%
\begin{equation}
\mathbb{E}\left[ g(a_{n})h(a_{n},B_{1,n});\tau (n)\in \lbrack m,n-m]\right]
\leq \frac{\varepsilon }{nc_{n}}  \label{Term1}
\end{equation}%
for all sufficiently large $n.$

Let $\mathcal{F}_{j+1,n}$ be the $\sigma $-algebra generated by the random
variables $X_{j+1},...,X_{n}$. Taking the conditional expectation with
respect to $\mathcal{F}_{j+1,n}$, we obtain%
\begin{equation*}
\mathbb{E}\left[ g(a_{n})h(a_{n},B_{1,n});\tau (n)=j\right] =\mathbb{E}\left[
\Upsilon _{n-j}\left( S_{j},B_{1,j}\right) ;\tau (j)=j\right]
\end{equation*}%
where%
\begin{equation}
_{n}\left( t,r\right) :=\mathbb{E}\left[
g(e^{-t}a_{n})h(e^{-t}a_{n},r+e^{-t}b_{n});L_{n}\geq 0\right] .
\label{Def_II}
\end{equation}%
We fix $\delta \in (0,1),$ write
\begin{equation*}
b_{n}=1+W_{n}+e^{-S_{n}}\tilde{T}_{n}
\end{equation*}%
and, for fixed $t\geq 0,r\geq 0$ introduce the function%
\begin{equation}
\varphi
_{t,r}(u,v,z):=g(e^{-t}e^{-z})h(e^{-t}e^{-z},r+e^{-t}(1+u)+e^{-t}e^{-z}v).
\label{DefFi}
\end{equation}%
By our conditions
\begin{eqnarray*}
0\leq \varphi (u,v,z)e^{\lambda _{1}z} &\leq &C^{2}\frac{e^{-\lambda _{1}t}}{%
\left( 1+e^{-t}e^{-z}+r+e^{-t}(1+u)+e^{-t}e^{-z}v\right) ^{\lambda _{2}}} \\
&\leq &C^{2}e^{-\lambda _{1}t}.
\end{eqnarray*}%
Thus,
\begin{equation}
\sup_{\left( u,v,z\right) \in \mathcal{W}\times \mathcal{T}\times \mathbb{R}%
_{+}}\varphi _{t,r}(u,v,z)e^{\lambda _{1}z}\leq C^{2}e^{-\lambda _{1}t}.
\label{Boundfi}
\end{equation}%
Since
\begin{equation*}
\varphi (W_{n},\tilde{T}%
_{n},S_{n})=g(e^{-t}e^{-S_{n}})h(e^{-t}e^{-S_{n}},r+e^{-t}(1+W_{n})+e^{-t}e^{-S_{n}}%
\tilde{T}_{n}),
\end{equation*}%
we may apply, basing on (\ref{Ginf}), (\ref{Hinf}) and (\ref{Boundfi}),
Lemma \ref{L_Newp41} to conclude that%
\begin{eqnarray*}
\lim_{n\rightarrow \infty }nc_{n}\Upsilon _{n}\left( t,r\right)
&=&\lim_{n\rightarrow \infty }nc_{n}\mathbb{E}\left[ \varphi _{t,r}(W_{n},%
\tilde{T}_{n},S_{n});L_{n}\geq 0\right] \\
&=&K\iiint \varphi _{t,r}(u,v,-z)\mathbb{P}^{+}\left( W_{\infty }\in
du\right) \mathbb{P}_{z}^{-}\left( T_{\infty }\in dv\right) V(z)dz\  \\
&=&:K\Pi \left( t,r\right) >0.
\end{eqnarray*}%
In view of (\ref{Boundfi}) and (\ref{AsymConditional}) there exists a
constant $C_{1}$ such that
\begin{equation*}
nc_{n}\Upsilon _{n}\left( t,r\right) \leq C^{2}e^{-\lambda _{1}t}\mathbb{E}%
\left[ e^{-\lambda _{1}S_{n}};L_{n}\geq 0\right] \leq C_{1}e^{-\lambda _{1}t}
\end{equation*}%
for all $n$. Hence, using the dominated convergence theorem we deduce that

\begin{equation}
\lim_{n\rightarrow \infty }nc_{n}\mathbb{E}\left[ g(a_{n})h(a_{n},B_{1,n});%
\tau (n)=j\right] =K\mathbb{E}\left[ \Pi \left( S_{j},B_{1,j}\right) ;\tau
(j)=j\right] =:K\Psi _{j}^{+}>0.  \label{Term2}
\end{equation}

We now fix $j$ and, making the substitution $j\leftrightarrows n-j$ anr
taking the expectation with respect to the $\sigma $-algebra $\mathcal{F}%
_{n-j,n}$, rewrite (\ref{Expect}) as
\begin{equation*}
\mathbb{E}\left[ g(a_{n})h(a_{n},B_{1,n});\tau (n)=n-j\right] =\mathbb{E}%
\left[ \Upsilon _{j}\left( S_{n-j},B_{1,n-j}\right) ;\tau (n-j)=n-j\right] .
\end{equation*}%
By (\ref{Estgh})%
\begin{eqnarray}
\Upsilon _{j}\left( t,r\right) &=&\mathbb{E}\left[
g(e^{-t}a_{j})h(e^{-t}a_{j},r+e^{-t}b_{j});L_{j}\geq 0\right]  \notag \\
&\leq &C^{2}\mathbb{E}\left[ \frac{e^{-\lambda _{1}t}a_{j}^{\lambda _{1}}}{%
\left( 1+e^{-t}a_{j}+r+e^{-t}b_{j}\right) ^{\lambda _{2}}};L_{j}\geq 0\right]
\notag \\
&\leq &C^{2}e^{(\lambda _{2}-\lambda _{1})t}\mathbb{E}\left[ e^{-\lambda
_{1}S_{j}};L_{j}\geq 0\right] \leq C^{2}e^{(\lambda _{2}-\lambda _{1})t}.
\label{Bond0}
\end{eqnarray}%
We fix, as before a $\delta \in \left( 0,1\right) $ and introduce the
notation%
\begin{equation*}
\Upsilon _{j}\left( S_{n},B_{1,n}\right) =\Upsilon _{j}\left(
S_{n},W_{n}+e^{-S_{n}}\tilde{T}_{n}\right) =:\varphi (W_{n},\tilde{T}%
_{n},S_{n}).
\end{equation*}%
Since%
\begin{equation*}
\sup_{\left( u,v,z\right) \in \mathcal{W}\times \mathcal{T}\times \mathbb{R}%
_{-}}\varphi (u,v,z)e^{-(\lambda _{2}-\lambda _{1})z}\leq C^{2}.
\end{equation*}%
by (\ref{Bond0}), we may, recalling (\ref{Ginf}) and (\ref{Hinf}) apply
Lemma \ref{L_Newp42} and conclude that%
\begin{eqnarray}
&&\lim_{n\rightarrow \infty }nc_{n}\mathbb{E}\left[ \Upsilon _{j}\left(
S_{n},W_{n}+e^{-S_{n}}\tilde{T}_{n}\right) ;\tau (n)=n\right]  \notag \\
&=&K\iiint \Upsilon _{j}\left( -z,u+e^{z}v\right) \mathbb{P}_{z}^{+}\left(
W_{\infty }\in du\right) \mathbb{P}^{-}\left( T_{\infty }\in dv\right)
U(z)dz(z)dz  \notag \\
\quad &:&=K\Psi _{j}^{-}>0.  \label{Term33}
\end{eqnarray}%
Combining (\ref{Term1}) - (\ref{Term33}) gives%
\begin{equation}
\lim_{n\rightarrow \infty }nc_{n}\mathbb{E}\left[ g(a_{n})h(a_{n},B_{1,n})%
\right] =K\left( \sum_{j=0}^{\infty }\Psi _{j}^{+}+\sum_{j=0}^{\infty }\Psi
_{j}^{-}\right) =:K_{g,h}>0.  \label{repKgh}
\end{equation}

The fact that $K_{g,h}<\infty $ follows from (\ref{Term1}), the estimates%
\begin{eqnarray*}
&&\sum_{j=0}^{m}\mathbb{E}\left[ g(a_{n})h(a_{n},B_{1,n});\tau (n)=j\right]
\\
&\leq &C\sum_{j=0}^{m}\mathbb{E}\left[ e^{(\lambda _{2}-\lambda
_{1})S_{j}};M_{j}<0\right] \mathbb{E}\left[ e^{-\lambda
_{1}S_{n-j}};L_{n-j}\geq 0\right]
\end{eqnarray*}%
and%
\begin{eqnarray*}
&&\sum_{j=0}^{m}\mathbb{E}\left[ g(a_{n})h(a_{n},B_{1,n});\tau (n)=n-j\right]
\\
&\leq &C\sum_{j=0}^{m}\mathbb{E}\left[ e^{(\lambda _{2}-\lambda
_{1})S_{n-j}};M_{n-j}<0\right] \mathbb{E}\left[ e^{-\lambda
_{1}S_{j}};L_{j}\geq 0\right] ,
\end{eqnarray*}%
valid for each fixed $m\in \lbrack 1,n/2]$, and Lemma \ref{L_Exponent_cond}.

This proves (\ref{GuivStatement}).

We now justify (\ref{GuivStatement2}). Clearly, for any $m\in \lbrack 1,n-1]$%
\begin{eqnarray*}
\mathbb{E}\left[ h(a_{n},B_{1,n});\tau (n)>m\right] &\leq &C\mathbb{E}\left[
\frac{1}{\left( 1+a_{n}+b_{n}\right) ^{\lambda _{2}}};\tau (n)>m\right] \\
&\leq &C\mathbb{E}\left[ e^{\lambda _{2}S_{\tau (n)}};\tau (n)>m\right] \\
&=&C\sum_{k=m}^{n}\mathbb{E}\left[ e^{\lambda _{2}S_{\tau (n)}};\tau (n)=k%
\right] \\
&=&C\sum_{k=m}^{n}\mathbb{E}\left[ e^{\lambda _{2}S_{\tau (n)}};\tau (k)=k%
\right] \mathbb{P}\left( L_{n-k}\geq 0\right) .
\end{eqnarray*}%
Recalling (\ref{AsymMin}) and (\ref{AsymConditional_max}) we conclude that,
for $\varepsilon >0$ there exists $m=m(\varepsilon )$ such that%
\begin{eqnarray}
&&\left( \sum_{k=m}^{\left[ n/2\right] \ \ \ \ }+\sum_{k=\left[ n/2\right] \
\ \ \ }^{n}\right) \mathbb{E}\left[ e^{\lambda _{2}S_{k}};\tau (k)=k\right]
\mathbb{P}\left( L_{n-k}\geq 0\right)  \notag \\
&\leq &\mathbb{P}\left( L_{n-\left[ n/2\right] }\geq 0\right) \sum_{k=m}^{%
\left[ n/2\right] }\frac{C_{1}}{kc_{k}}+\frac{C_{1}}{nc_{n}}\sum_{j=0}^{%
\left[ n/2\right] }\mathbb{P}\left( L_{j}\geq 0\right)  \notag \\
&\leq &\varepsilon \mathbb{P}\left( L_{n}\geq 0\right) +\frac{C_{1}}{nc_{n}}n%
\mathbb{P}\left( L_{n}\geq 0\right) \leq 2\varepsilon \mathbb{P}\left(
L_{n}\geq 0\right) .  \label{Term4}
\end{eqnarray}%
On the other hand, for each fixed $j\in \lbrack 0,m]$%
\begin{equation*}
\mathbb{E}\left[ h(a_{n},B_{1,n});\tau (n)=j\right] =\mathbb{E}\left[ \Theta
_{n-j}(S_{j},B_{1,j});\tau (j)=j\right]
\end{equation*}%
where
\begin{equation*}
\Theta _{n}(t,r):=\mathbb{E}\left[ h(e^{-t}a_{n},r+e^{-t}b_{n});L_{n}\geq 0%
\right] .
\end{equation*}%
We know that, as $n\rightarrow \infty $ $S_{n}\rightarrow +\infty $ and $%
b_{n}\rightarrow 1+W_{\infty }$\thinspace\ $\mathbb{P}^{+}-$a.s. Since $%
h(x,y)$ is continuous and uniformly bounded for $x\geq 0,y\geq 0,$ it
follows that
\begin{equation*}
h(e^{-t}a_{n},r+e^{-t}b_{n})\rightarrow h(0,r+e^{-t}W_{\infty })\quad
\mathbb{P}^{+}-a.s.
\end{equation*}%
as $n\rightarrow \infty $. \ Thus, we may apply Lemma 2.5\ in \cite{4h} to
conclude that, as $n\rightarrow \infty $%
\begin{equation*}
\Theta _{n}(t,r)\sim \mathbb{E}^{+}\left[ h(0,r+e^{-t}W_{\infty })\right]
\mathbb{P}\left( L_{n}\geq 0\right) .
\end{equation*}%
Recalling (\ref{AsymMin}) and using the dominated convergence theorem we
deduce that, for any fixed $j$%
\begin{eqnarray*}
&&\lim_{n\rightarrow \infty }\frac{\mathbb{E}\left[
h(e^{-t}a_{n},r+e^{-t}b_{n});\tau (n)=j\right] }{\mathbb{P}\left( L_{n}\geq
0\right) } \\
&&\qquad =\lim_{n\rightarrow \infty }\mathbb{E}\left[ \frac{\Theta
_{n-j}(S_{j},B_{1,j})}{\mathbb{P}\left( L_{n-j}\geq 0\right) };\tau (j)=j%
\right] \frac{\mathbb{P}\left( L_{n-j}\geq 0\right) }{\mathbb{P}\left(
L_{n}\geq 0\right) } \\
&&\qquad =\left( \mathbb{E\times E}^{+}\right) \left[ h(0,B_{1,j}+a_{j}W_{%
\infty });\tau (j)=j\right] ,
\end{eqnarray*}%
where we have used the notation $\mathbb{E\times E}^{+}$ to indicate that $%
B_{1,j}$ and $a_{j}$ are distributed according to the measure $\mathbb{P}$
while $W_{\infty }$ is distributed according to the measure $\mathbb{P}^{+}.$
Since $\varepsilon >0$ in (\ref{Term4}) may be selected arbitrary small, we
see that%
\begin{eqnarray}
\lim_{n\rightarrow \infty }\frac{\mathbb{E}\left[ h(a_{n},B_{1,n})\right] }{%
\mathbb{P}\left( L_{n}\geq 0\right) } &=&\sum_{j=0}^{\infty }\left( \mathbb{%
E\times E}^{+}\right) \left[ h(0,B_{1,j}+a_{j}W_{\infty });\tau (j)=j\right]
\notag \\
&=&:K_{h}\in \left( 0,\infty \right) .  \label{Dop24}
\end{eqnarray}

Lemma \ref{L_Guiv1} is proved.

\section{Proofs of the main results}

\label{sec_nt2}

First we prove points 1) of all three theorems. We know from Corollary \ref%
{C_fractional} that
\begin{eqnarray*}
\mathbf{P}\left( A_{n-N}(n)\right) &=&\mathbf{E}\left[ \mathcal{T}_{n-N,n}%
\right] =\mathbf{E}\left[ \frac{a_{n-N}}{a_{n}+b_{n}-b_{n-N+1}}\frac{a_{n}}{%
a_{n}+b_{n}}\right] \\
&=&\mathbf{E}\left[ \frac{e^{S_{n}-S_{n-N}}}{%
\sum_{k=n-N+1}^{n}e^{S_{n}-S_{k}}}\frac{1}{\sum_{k=0}^{n}e^{S_{n}-S_{k}}}%
\right] .
\end{eqnarray*}%
Making the substitution%
\begin{equation*}
\hat{S}_{r}:=S_{n}-S_{n-r},r=0,1,...,n,
\end{equation*}%
and using the equality $\left\{ \hat{S}_{r},r=0,1,...,n\right\} \overset{d}{=%
}\left\{ S_{r},r=0,1,...,n\right\} $ we obtain%
\begin{equation}
\mathbf{P}\left( A_{n-N}(n)\right) =\mathbf{E}\left[ \frac{e^{\hat{S}_{N}}}{%
\sum_{r=0}^{N-1}e^{\hat{S}_{r}}}\frac{1}{\sum_{r=0}^{n}e^{\hat{S}_{r}}}%
\right] =\mathbf{E}\left[ \frac{e^{S_{N}}}{\sum_{r=0}^{N-1}e^{S_{r}}}\frac{1%
}{\sum_{r=0}^{n}e^{S_{r}}}\right] .  \label{BasicEqua}
\end{equation}%
Since $\mathbf{E}X$ $<0$,
\begin{equation}
\sum_{r=0}^{\infty }e^{S_{r}}<\infty \quad \mathbf{P}\text{-a.s.}
\label{AsConvergence0}
\end{equation}%
Therefore, we may apply the dominated convergence theorem to conclude that%
\begin{equation*}
\lim_{n\rightarrow \infty }\mathbf{P}\left( A_{n-N}(n)\right) =\mathbf{E}%
\left[ \frac{e^{S_{N}}}{\sum_{r=0}^{N-1}e^{S_{r}}}\frac{1}{%
\sum_{r=0}^{\infty }e^{S_{r}}}\right] =:r_{N}>0.
\end{equation*}

\subsection{Proof of point 2) in Theorem \protect\ref{T_strongly} (strongly
subcritical case)}

For $x,y\geq 0$ introduce the functions
\begin{equation*}
\Lambda _{i}\left( x,y\right) :=\mathbf{E}\left[ \frac{1}{%
1+y+x\sum_{r=0}^{i}e^{S_{r}}}\right]
\end{equation*}%
and
\begin{equation}
h_{i}\left( x,y\right) :=\frac{1}{1+y}\Lambda _{i}\left( x,y\right) .
\label{h_representation}
\end{equation}%
Clearly,
\begin{equation}
h_{i}\left( 0,y\right) =\frac{1}{1+y}\Lambda _{i}\left( 0,y\right) =\frac{1}{%
(1+y)^{2}}\text{ and }h_{i}\left( x,y\right) \leq \frac{1}{1+x+y},\ x,y\geq
0.  \label{h_indep}
\end{equation}%
By (\ref{AsConvergence0})%
\begin{equation}
h_{\infty }(x,y):=\lim_{i\rightarrow \infty }h_{i}\left( x,y\right) =\frac{1%
}{1+y}\mathbf{E}\left[ \frac{1}{1+y+x\sum_{r=0}^{\infty }e^{S_{r}}}\right]
>0.  \label{Positivite}
\end{equation}

Using (\ref{BasicEqua}) and making, with $\delta =1$ the standard change of
measure (\ref{Change_delta}) for the random sequence $\left\{
S_{r},r=0,1,...,n-i\right\} ,$ we obtain
\begin{eqnarray}
\mathbf{P}\left( A_{i}(n)\right) &=&\gamma ^{n-i}\mathbb{E}\left[ \frac{1}{%
\sum_{r=0}^{n-i-1}e^{S_{r}}}\Lambda _{i}\left(
e^{S_{n-i}},\sum_{r=1}^{n-i-1}e^{S_{r}}\right) \right]  \notag \\
&=&\gamma ^{n-i}\mathbb{E}\left[ h_{i}\left(
e^{S_{n-i}},\sum_{r=1}^{n-i-1}e^{S_{r}}\right) \right].  \label{changeni}
\end{eqnarray}%
Note that after the change of measure the independent increments $%
X_{1},...,X_{n-i}$ of the random walk $\mathbf{S}$ are distributed according
to the law $\mathbb{P}(dx):=e^{x}\mathbf{P}\left( dx\right) $ while the
independent increments $X_{n-i+1},...,X_{n}$ are distributed according to
the law $\mathbf{P}\left( dx\right) $. Since
\begin{equation*}
\mathbb{E}X=\frac{\mathbf{E}Xe^{X}}{\mathbf{E}X}<0,
\end{equation*}%
it follows that, as $N\rightarrow \infty $
\begin{equation*}
\sum_{r=1}^{N-1}e^{S_{r}}\rightarrow \sum_{r=1}^{\infty }e^{S_{r}}<\infty
\text{ \ }\mathbb{P}\text{-a.s.}
\end{equation*}%
and $e^{S_{N}}\rightarrow 0$ $\mathbb{P-}a.s.$ These estimates, (\ref%
{AsConvergence0}) and the dominated convergence theorem give%
\begin{eqnarray*}
\lim_{n-i\rightarrow \infty }\gamma ^{-(n-i)}\mathbf{P}\left(
A_{i}(n)\right) &=&\mathbb{E}\left[ \lim_{n-i\rightarrow \infty }\frac{1}{%
\sum_{r=0}^{n-i-1}e^{S_{r}}}\Lambda _{i}\left(
e^{S_{n-i}},\sum_{r=1}^{n-i-1}e^{S_{r}}\right) \right] \\
&=&\mathbb{E}\left[ \frac{1}{\sum_{r=0}^{\infty }e^{S_{r}}}\Lambda
_{i}\left( \lim_{n-i\rightarrow \infty }e^{S_{n-i}},\lim_{n-i\rightarrow
\infty }\sum_{r=1}^{n-i-1}e^{S_{r}}\right) \right] \\
&=&\mathbb{E}\left[ \frac{1}{\sum_{r=0}^{\infty }e^{S_{r}}}\Lambda
_{i}\left( 0,\sum_{r=1}^{\infty }e^{S_{r}}\right) \right] \\
&=&\mathbb{E}\left[ \frac{1}{\left( \sum_{r=0}^{\infty }e^{S_{r}}\right) ^{2}%
}\right] .
\end{eqnarray*}

Point 2) of Theorem \ref{T_strongly} is proved.

\subsection{Proof of point 2) in Theorem \protect\ref{T_intermediate}
(intermediate subcritical case)}

\subsubsection{The case of fixed $i$}

We again make the change of measure (\ref{Change_delta}) with $\delta =1$
and rewrite (\ref{changeni}) as
\begin{equation}
\mathbf{P}\left( A_{i}(n)\right) =\gamma ^{n-i}\mathbb{E}\left[ h_{i}\left(
e^{-\bar{S}_{n-i}},\sum_{r=1}^{n-i-1}e^{-\bar{S}_{r}}\right) \right] ,
\label{new1}
\end{equation}%
where $\left\{ \bar{S}_{r},r=0,1,...,n-i\right\} =\left\{
-S_{r},r=0,1,...,n-i\right\} $. Since $\bar{X}\overset{d}{=}-X$ with respect
to the measure $\mathbb{P}$, it follows that
\begin{equation*}
\lim_{n\rightarrow \infty }\mathbb{P}(\bar{S}_{r}>0)=1-\lim_{n\rightarrow
\infty }\mathbb{P}(S_{n}<0)=1-\rho .
\end{equation*}%
This relation, Lemma \ref{L_min_max}, (\ref{Dop24}) and (\ref{h_indep})
imply for $\bar{L}_{n}:=\min (\bar{S}_{0},\bar{S}_{1},\ldots ,\bar{S}_{n})$%
\begin{eqnarray*}
\lim_{n\rightarrow \infty }\frac{\mathbf{P}\left( A_{i}(n)\right) }{\gamma
^{n-i}\mathbb{P}\left( \bar{L}_{n}\geq 0\right) } &=&\sum_{j=0}^{\infty
}\left( \mathbb{E\times E}^{+}\right) \left[ h_{i}(0,B_{1,j}+a_{j}W_{\infty
});\tau (j)=j\right] \\
&=&\sum_{j=0}^{\infty }\left( \mathbb{E\times E}^{+}\right) \left[ \frac{1}{%
\left( B_{1,j}+a_{j}W_{\infty }\right) ^{2}};\tau (j)=j\right] >0.
\end{eqnarray*}

\subsubsection{The case $\min \left( i,n-i\right) \rightarrow \infty $}

We write an equlvalent form of (\ref{new1})
\begin{equation}
\mathbf{P}\left( A_{i}(n)\right) =\gamma ^{n-i}\mathbb{E}\left[ h_{i}\left(
e^{-\bar{S}_{n-i}},\sum_{r=1}^{n-i-1}e^{-\bar{S}_{r}}\right) \right] =%
\mathbb{E}\left[ \frac{1}{\sum_{r=1}^{n-i-1}e^{-\bar{S}_{r}}}\Lambda
_{i}\left( e^{-\bar{S}_{n-i}},\sum_{r=1}^{n-i-1}e^{-\bar{S}_{r}}\right) %
\right] .  \label{Dop1}
\end{equation}%
Observe now that, for any $1\leq m<i$%
\begin{eqnarray}
0 &\leq &\Lambda _{m}\left( x,y\right) -\Lambda _{i}\left( x,y\right) =%
\mathbf{E}\left[ \frac{x\sum_{r=m+1}^{i}e^{S_{r}}}{1+y+x%
\sum_{r=0}^{m}e^{S_{r}}}\frac{1}{1+y+x\sum_{r=0}^{i}e^{S_{r}}}\right]  \notag
\\
&\leq &\mathbf{E}\left[ \frac{1}{1+y+x\sum_{r=0}^{m}e^{S_{r}}}\frac{%
\sum_{r=m+1}^{i}e^{S_{r}}}{\sum_{r=0}^{i}e^{S_{r}}}\right] \leq \frac{1}{%
1+x+y}\mathbf{E}\left[ \frac{\sum_{r=m+1}^{i}e^{S_{r}}}{%
\sum_{r=0}^{i}e^{S_{r}}}\right] .  \label{difLambda}
\end{eqnarray}%
Hence, using (\ref{GuivStatement2}) with $h(x,y)=\left( 1+x+y\right) ^{-1}$
and recalling Lemma \ref{L_min_max} we conclude that
\begin{eqnarray}
\Gamma _{m}(i,n-i):= &&\mathbb{E}\left[ h_{m}\left( e^{-\bar{S}%
_{n-i}},\sum_{r=1}^{n-i-1}e^{-\bar{S}_{r}}\right) \right] -\mathbb{E}\left[
h_{i}\left( e^{-\bar{S}_{n-i}},\sum_{r=1}^{n-i-1}e^{-\bar{S}_{r}}\right) %
\right]  \notag \\
&\leq &\mathbb{E}\left[ \frac{1}{1+\sum_{r=1}^{n-i-2}e^{-\bar{S}_{r}}+e^{-%
\bar{S}_{n-j-1}}}\right] \mathbf{E}\left[ \frac{\sum_{r=m+1}^{i}e^{S_{r}}}{%
\sum_{r=0}^{i}e^{S_{r}}}\right]  \notag \\
&=&\mathbb{E}\left[ h\left( e^{-\bar{S}_{n-j-1}},\sum_{r=1}^{n-i-2}e^{-\bar{S%
}_{r}}\right) \right] \mathbf{E}\left[ \frac{\sum_{r=m+1}^{i}e^{S_{r}}}{%
\sum_{r=0}^{i}e^{S_{r}}}\right]  \notag \\
&\leq &C\mathbb{P}\left( \bar{L}_{n-i}\geq 0\right) \mathbf{E}\left[ \frac{%
\sum_{r=m+1}^{i}e^{S_{r}}}{\sum_{r=0}^{i}e^{S_{r}}}\right] .  \label{dop2}
\end{eqnarray}%
Combining this estimate with (\ref{AsConvergence0}) we obtain%
\begin{equation*}
\lim_{m\rightarrow \infty }\limsup_{\min (i,n-i)\rightarrow \infty }\frac{%
\Gamma _{m}(i,n-i)}{\mathbb{P}\left( \bar{L}_{n-i}\geq 0\right) }\leq
C\lim_{m\rightarrow \infty }\mathbf{E}\left[ \frac{\sum_{r=m+1}^{\infty
}e^{S_{r}}}{\sum_{r=0}^{\infty }e^{S_{r}}}\right] =0.
\end{equation*}%
On the other hand, according to (\ref{Dop24}) and (\ref{h_indep})
\begin{eqnarray}
\lim_{n-i\rightarrow \infty }\frac{\mathbb{E}\left[ h_{m}\left( e^{-\bar{S}%
_{n-i}},\sum_{r=1}^{n-i-1}e^{-\bar{S}_{r}}\right) \right] }{\mathbb{P}\left(
\bar{L}_{n-i}\geq 0\right) } &=&\sum_{j=0}^{\infty }\left( \mathbb{E\times E}%
^{+}\right) \left[ h_{m}(0,B_{1,j}+a_{j}W_{\infty });\tau (j)=j\right]
\notag \\
&=&\sum_{j=0}^{\infty }\left( \mathbb{E\times E}^{+}\right) \left[ \frac{1}{%
\left( B_{1,j}+a_{j}W_{\infty }\right) ^{2}};\tau (j)=j\right] .
\label{dop3}
\end{eqnarray}%
for any fixed $m$. Since $m$ may be selected arbitrary large, it follows
from (\ref{Dop1}) -- (\ref{dop3}) that
\begin{equation}
\lim_{\min (i,n-i)\rightarrow \infty }\frac{\mathbf{P}\left( A_{i}(n)\right)
}{\gamma ^{n-i}\mathbb{P}\left( \bar{L}_{n-i}\geq 0\right) }%
=\sum_{j=0}^{\infty }\left( \mathbb{E\times E}^{+}\right) \left[ \frac{1}{%
\left( B_{1,j}+a_{j}W_{\infty }\right) ^{2}};\tau (j)=j\right] .
\end{equation}%
Theorem \ref{T_intermediate} is proved.

\subsection{Proofs of points 2) and 3) in Theorem \protect\ref{T_weakly}
(weakly subcritical case)}

\bigskip

\subsubsection{The case of fixed $i$}

Using (\ref{BasicEqua}) with $N=n-i$ and changing the measure as in (\ref%
{Change_delta}) with $\delta =\beta $ we obtain
\begin{eqnarray*}
\mathbf{P}\left( A_{i}(n)\right) &=&\mathbf{E}\left[ e^{S_{n-i}}h_{i}\left(
e^{S_{n-i}},\sum_{r=1}^{n-i-1}e^{S_{r}}\right) \right] \\
&=&\gamma ^{n-i}\mathbb{E}\left[ e^{-(1-\beta )\bar{S}_{n-i}}h_{i}\left( e^{-%
\bar{S}_{n-i}},\sum_{r=1}^{n-i-1}e^{-\bar{S}_{r}}\right) \right] .
\end{eqnarray*}%
Using (\ref{GuivStatement}) with $g(x)=x^{1-\beta }$ and $h(x,y)=h_{i}(x,y)$
as in (\ref{h_representation}) and recalling (\ref{repKgh}) we get%
\begin{equation*}
\lim_{n\rightarrow \infty }(n-i)c_{n-i}\frac{\mathbf{P}\left(
A_{i}(n)\right) }{\gamma ^{n-i}}=K\left( \sum_{j=0}^{\infty }\Psi
_{ij}^{+}+\sum_{j=0}^{\infty }\Psi _{ij}^{-}\right) .
\end{equation*}%
The summands at the right-hand side have the form
\begin{equation}
\Psi _{ij}^{+}:=\mathbb{E}\left[ \Pi _{i}\left( S_{j},B_{1,j}\right) ;\tau
(j)=j\right]  \label{Psiij1}
\end{equation}%
with%
\begin{equation}
\Pi _{i}\left( t,r\right) :=\iiint \varphi _{i,t,r}(u,v,-z)\mathbb{P}%
^{+}\left( W_{\infty }\in du\right) \mathbb{P}_{z}^{-}\left( T_{\infty }\in
dv\right) V(z)dz  \label{Def_P}
\end{equation}%
and (compare with (\ref{DefFi}))%
\begin{eqnarray*}
\varphi _{i,t,r}(u,v,z):= &&e^{-(1+\beta
)(t+z)}h_{i}(e^{-(t+z)},r+e^{-t}(1+u)+e^{-(t+z)}v) \\
&=&e^{-(1+\beta )(t+z)}\frac{1}{1+r+e^{-t}(1+u)+e^{-(t+z)}v} \\
&&\times \mathbf{E}\left[ \frac{1}{1+r+e^{-t}(1+u)+e^{-(t+z)}v+e^{-(t+z)}%
\sum_{r=0}^{i-1}e^{S_{r}}}\right]
\end{eqnarray*}%
and%
\begin{equation}
\Psi _{ij}^{-}:=\iiint \Upsilon _{ij}\left( -z,u+e^{z}v\right) \mathbb{P}%
_{z}^{+}\left( W_{\infty }\in du\right) \mathbb{P}^{-}\left( T_{\infty }\in
dv\right) U(z)dz,  \label{PSii2}
\end{equation}%
where (see (\ref{Def_II}))%
\begin{equation}
\Upsilon _{ij}\left( t,r\right) :=e^{-(1-\beta )t}\mathbb{E}\left[
e^{-(1-\beta )S_{j}}h_{i}\left(
e^{-t}e^{-S_{j}},r+e^{-t}\sum_{k=0}^{j-1}e^{-S_{k}}\right) ;L_{j}\geq 0%
\right] .  \label{DefvaveII}
\end{equation}%
This proves point 2) of Theorem \ref{T_weakly} for any fixed $i$.

\subsection{The case $\min \left( i,n-i\right) \rightarrow \infty $}

We consider
\begin{equation*}
\mathbb{E}\left[ e^{-(1-\beta )\bar{S}_{n-i}}h_{i}\left( e^{-\bar{S}%
_{n-i}},\sum_{r=1}^{n-i-1}e^{-\bar{S}_{r}}\right) \right] .
\end{equation*}%
It follows the same as before, that, the limit
\begin{equation}
\lim_{n-i\rightarrow \infty }(n-i)c_{n-i}\mathbb{E}\left[ e^{-(1-\beta )\bar{%
S}_{n-i}}h_{m}\left( e^{-\bar{S}_{n-i}},\sum_{r=1}^{n-i-1}e^{-\bar{S}%
_{r}}\right) \right] =K\left( \sum_{j=0}^{\infty }\Psi
_{mj}^{+}+\sum_{j=0}^{\infty }\Psi _{mj}^{-}\right)  \label{Approximation2}
\end{equation}%
exists for each fixed $m.$ We\ see by (\ref{Positivite}) that%
\begin{equation*}
\lim_{m\rightarrow \infty }\varphi _{m,t,r}(u,v,z)=e^{-(1+\beta
)(t+z)}h_{\infty }(e^{-(t+z)},r+e^{-t}(1+u)+e^{-(t+z)}v)>0
\end{equation*}%
and%
\begin{equation*}
\lim_{m\rightarrow \infty }\Upsilon _{mj}\left( t,r\right) =e^{-(1-\beta )t}%
\mathbb{E}\left[ e^{-(1-\beta )S_{j}}h_{\infty }\left(
e^{-t}e^{-S_{j}},r+e^{-t}\sum_{k=0}^{j-1}e^{-S_{k}}\right) ;L_{j}\geq 0%
\right] >0.
\end{equation*}%
Recalling (\ref{Psiij1})-(\ref{DefvaveII}) we conclude that the limits
\begin{equation*}
\Psi _{\infty j}^{+}:=\lim_{m\rightarrow \infty }\Psi _{mj}^{+}\quad \text{%
and \quad }\Psi _{\infty j}^{-}:=\lim_{m\rightarrow \infty }\Psi _{mj}^{-}
\end{equation*}%
exist for each $j$ and are finite and positive.

Using (\ref{difLambda}) we obtain for $m<i$
\begin{eqnarray*}
0\leq h_{m}\left( x,y\right) -h_{i}\left( x,y\right) &=&\frac{1}{1+y}\left(
\Lambda _{m}\left( x,y\right) -\Lambda _{i}\left( x,y\right) \right) \\
&\leq &\frac{1}{1+x+y}\mathbf{E}\left[ \frac{\sum_{r=m+1}^{i}e^{S_{r}}}{%
\sum_{r=0}^{i}e^{S_{r}}}\right] .
\end{eqnarray*}%
Hence it follows that
\begin{eqnarray*}
\Gamma _{m}^{\prime }(i,n-i) &:&=\mathbb{E}\left[ e^{-(1-\beta )\bar{S}%
_{n-i}}\left\{ h_{m}\left( e^{-\bar{S}_{n-i}},\sum_{r=1}^{n-i-1}e^{-\bar{S}%
_{r}}\right) -h_{i}\left( e^{-\bar{S}_{n-i}},\sum_{r=1}^{n-i-1}e^{-\bar{S}%
_{r}}\right) \right\} \right] \\
&\leq &\mathbb{E}\left[ \frac{e^{-(1-\beta )\bar{S}_{n-i}}}{%
\sum_{r=0}^{n-i-1}e^{S_{r}}}\right] \mathbf{E}\left[ \frac{%
\sum_{r=m+1}^{i}e^{S_{r}}}{\sum_{r=0}^{i}e^{S_{r}}}\right] .
\end{eqnarray*}

Applying now (\ref{GuivStatement}) with $g(x):=x^{1-\beta }$ and $%
h(x,y):=(1+x+y)^{-1}$ we conclude that there exists a constant $C_{1}$ such
that
\begin{equation*}
\Gamma _{m}^{\prime }(i,n-i)\leq C_{1}(n-i)c_{n-i}\mathbf{E}\left[ \frac{%
\sum_{r=m+1}^{i}e^{S_{r}}}{\sum_{r=0}^{i}e^{S_{r}}}\right] .
\end{equation*}%
By the dominated convergence theorem we deduce that
\begin{equation*}
\lim_{m\rightarrow \infty }\limsup_{\min (i,n-i)\rightarrow \infty }\frac{%
\Gamma _{m}^{\prime }(i,n-i)}{(n-i)c_{n-i}}\leq C_1\lim_{m\rightarrow \infty
}\mathbf{E}\left[ \frac{\sum_{r=m+1}^{\infty }e^{S_{r}}}{\sum_{r=0}^{\infty
}e^{S_{r}}}\right] =0.
\end{equation*}
Since $m$ in (\ref{Approximation2}) may be selected arbitrary large, it
follows that
\begin{equation*}
\lim_{\min (i,n-i)\rightarrow \infty }\frac{\mathbf{P}\left( A_{i}(n)\right)
}{(n-i)c_{n-i}\gamma ^{n-i}}=K\left( \sum_{j=0}^{\infty }\Psi _{\infty
j}^{+}+\sum_{j=0}^{\infty }\Psi _{\infty j}^{-}\right) .
\end{equation*}

Theorem \ref{T_weakly} is proved.

\textbf{Acknowledgement} This work is supported by the Russian Science
Foundation under the grant 19-11-00111.

\end{document}